\documentclass[12pt,reqno]{amsart}
\usepackage[latin1]{inputenc}
\usepackage[T1]{fontenc}
\usepackage{mathptmx}
\usepackage{amssymb}

\usepackage{color}

\usepackage{paralist}

\newtheorem{theorem}{Theorem}[section]
\newtheorem{lemma}[theorem]{Lemma}
\newtheorem{corollary}[theorem]{Corollary}

\theoremstyle{definition}
\newtheorem{example}[theorem]{Example}

\numberwithin{equation}{section}

\newcommand{\set}[1]{\left\{\,#1\,\right\}}  
\newcommand{\with}{\ \vrule\ }  
\newcommand{\defa}{:=} 

\newcommand{\NN}{\mathbb{N}}
\newcommand{\ZZ}{\mathbb{Z}}
\newcommand{\QQ}{\mathbb{Q}}
\newcommand{\kk}{\mathbb{K}}

\newcommand{\LS}{\ZZ[\![t]\!][t^{-1}]}

\newcommand{\nt}{\tilde{h}}
\newcommand{\kt}{\tilde{k}}

\newcommand{\Q}[1]{\mathcal{N}(#1)}

\begin{document}
\title[Hilbert series of modules over graded polynomial rings]{Hilbert series of modules over \\ positively graded polynomial rings}

\author{Lukas Katth\"an}

\address{Goethe-Universit\"at Frankfurt, Institut f\"ur Mathematik, 60325 Frankfurt am Main, Germany} \email{katthaen@math.uni-frankfurt.de}

\author{Julio Jos\'e Moyano-Fern\'andez}

\address{Universitat Jaume I, Campus de Riu Sec, Departamento de Matem\'aticas \& Institut Universitari de Matem\`atiques i Aplicacions de Castell\'o, 12071
Caste\-ll\'on de la Plana, Spain} \email{moyano@uji.es}

\author{Jan Uliczka}

\address{Universit\"at Osnabr\"uck, FB Mathematik/Informatik, 49069
Osnabr\"uck, Germany} \email{juliczka@uos.de}

\thanks{The first and second authors were partially supported by  
the German Research Council DFG-GRK~1916. The second author was further supported by the Spanish Government Ministerio de Econom\'ia y Competitividad (MINECO), grants MTM2012-36917-C03-03 and MTM2015-65764-C3-2-P, as well as by Universitat Jaume I, grant P1-1B2015-02}

\begin{abstract}
In this note, we give examples of formal power series satisfying certain conditions that cannot be realized as Hilbert series of finitely generated modules.
This answers to the negative a question raised in a recent article by the second and the third author.
On the other hand, we show that the answer is positive after multiplication with a scalar.
\end{abstract}

\maketitle

\section{Introduction}

Let $\kk$ be a field, and let $R=\kk[X_1, \ldots , X_n]$ be the positively $\ZZ$-graded polynomial ring with $\deg{X_i}=d_i\geq 1$ for every $i=1,\ldots ,n$. Consider a finitely generated graded $R$--module 
$M=\bigoplus_k M_k$ over $R$. The graded components $M_k$ of $M$ are finitely dimensional $\kk$-vector spaces, and, since $R$ is positively graded, $M_k=0$ for $k \ll 0$.
The formal Laurent series
\[
H_M(t):=\sum_{k \in \ZZ} (\dim_K M_k)t^k \in \LS
\]
is called the Hilbert series of $M$.
Obviously every coefficient of this series is nonnegative.
Moreover, it is well-known that $H_M(t)$ can be written as a rational function with denomi\-nator $(1-t^{d_1})\dotsm(1-t^{d_n})$. 
In fact, in the standard graded case (i.e. $d_1 = \dotsb = d_n = 1$) these two properties characterize the Hilbert series of finitely generated $R$-modules among the formal Laurent series $\LS$, cf. Uliczka \cite[Cor. 2.3]{U}. 

In the non-standard graded case, the situation is more involved. A characterization of Hilbert series was obtained  by the second and third author in \cite{MU}:
\begin{theorem}[Moyano-Uliczka]\label{thm:jan}
Let $P(t) \in \LS$ be a formal Laurent series which is rational with denominator $(1-t^{d_1})\dotsm(1-t^{d_n})$.
Then $P$ can be realized as Hilbert series of some finitely generated $R$-module if and only if it can be written in the form
\begin{equation} \label{2eq1}
	P(t) = \sum_{I \subseteq \{1, \ldots , n\}} \frac{Q_I(t)}{\prod_{i \in I} \left(1-t^{d_i} \right)} 
\end{equation}
with nonnegative  $Q_I \in \ZZ[t,t^{-1}]$.
\end{theorem}
However, it remained an open question in \cite[Remark 2.3]{MU} if the condition of the Theorem is satisfied by \emph{every} rational function with the given denominator and nonnegative coefficients. In this paper we answer this question to the negative.
In Section \ref{section:examples} we provide examples of rational functions that do not admit a decomposition \eqref{2eq1} and are thus not realizable as Hilbert series.
On the other hand, we show the following in Corollary \ref{cor:rat} and Theorem \ref{thm:2var}:
\begin{theorem}
	Assume that the degrees $d_1, \dotsc, d_n$ are pairwise either coprime or equal.
	Then the following holds:
	\begin{enumerate}
	\item If $n=2$, then every rational function $P(t) \in \LS$ with the given denominator and nonnegative coefficients admits a decomposition as in \eqref{2eq1}
	\item In general, the same still holds \emph{up to multiplication with a scalar}.
	\end{enumerate}
\end{theorem}
In particular, there is a formal Laurent series $P(t)$ with integral coefficients such that $2P(t)$, but not $P(t)$,  is the Hilbert series of a finitely generated graded $R$--module, cf.~Exam\-ple \ref{ex3.1}.
Moreover, we will provide an example (Example \ref{ex3.3}) showing that the conclusion does not hold without the assumption on the degrees being pairwise coprime.

\section{Proofs of the main results}
As general references for further details about Hilbert series the reader is referred to Bruns and Herzog~ \cite{BH}. 
Furthermore, we are going to use some well-known facts about quasipolynomials and power series expansions of rational functions. For details about these topics, we refer the reader 
to Chapter 4 of Stanley~\cite{stanley}.

We first show three lemmas before we present the proof of our main results.
The following notation will be useful.
For $\delta \in \NN$ and $0 \leq j \leq \delta -1$ set
	\[ e_{\delta,j}(h) \defa \begin{cases}
	1 &\text{ if } h \equiv j \mod \delta, \\
	0 & \text{ otherwise}.
	\end{cases}\]
Obviously, the functions $e_{\delta,0},\dotsc,e_{\delta,\delta-1}$ form a basis of the space of $\delta$-periodic functions $\NN\rightarrow\QQ$.

\begin{lemma}\label{lem:periodic}
	Let $c_1, \dotsc, c_r: \NN \rightarrow \QQ$ be periodic functions of periods $\delta_1, \dotsc, \delta_r$, such that their sum takes nonnegative values.
	Then there exist nonnegative periodic functions $\tilde{c}_1, \dotsc, \tilde{c}_r: \NN \rightarrow \QQ$ of the same periods such that $\sum_i c_i = \sum_i \tilde{c}_i$.
	Moreover, if the sum of the $c_i$ takes nonnegative integral values, then the $\tilde{c}_i$ can be chosen to be integral valued. 
\end{lemma}
\begin{proof}
	Let us define the coefficients $\mu(i,j)$ by requiring
	\[ c_i = \sum_{j=0}^{\delta_i-1} \mu(i,j) e_{\delta_i,j}. \]
	For each $i>1$, let $m_i$ be the minimum of the $\mu({i,1}), \dotsc, \mu({i,\delta_i})$
	and choose a $k_i$ such that $m_i = \mu(i, k_i)$.
	\newcommand{\mut}{\tilde{\mu}}
	Set $\mut({i,j}) \defa \mu({i,j}) + m_i$ for $1 < i \leq r$, $\mut({1,j}) \defa \mu({1,j}) - \sum_i m_i$ 
	and define $\tilde{c}_i \defa \sum_{j=0}^{\delta_i-1} \mut(i,j) e_{\delta_i,j}$.
	Using the relation  
	\[
	\sum_{j=0}^{\delta-1} e_{\delta,j} = \sum_{j=0}^{\delta'-1} e_{\delta',j},
	\]
	which holds for all $\delta,\delta' \in \NN$, one easily sees that
	\begin{equation*} 
		\sum_{i=1}^r c_i = \sum_{i=1}^r \sum_{j=0}^{\delta_i - 1} \mu({i,j}) e_{\delta_i,j} = \sum_{i=1}^r \sum_{j=0}^{\delta_i - 1} \mut({i,j}) e_{\delta_i,j} = \sum_{i=1}^r \tilde{c}_i.
	\end{equation*}
	By construction we have $\mut({i,j}) \geq 0$ for $i > 1$ and all $j$, and we claim that also $\mut({1,j}) \geq 0$ for all $j$. To prove this, assume for contrary that there exists an index $j_0$ such that $\mut({1,j_0}) < 0$.
	Note that by construction $\mut(i, k_i) = 0$ for $1 < i \leq r$. By the Chinese remainder theorem there exists an $0 \leq h < \delta_1 \delta_2 \dotsm \delta_r$ such that $h \equiv j_0 \mod \delta_1$ and $h \equiv k_i \mod \delta_i$ for $i > 1$.
	Then 
	\begin{align*}
	\sum_{i=1}^r c_i(h) &= \sum_{i=1}^r \tilde{c}_i(h) \\
	&= \mut(1,j_0) + \mut(2, k_2) + \mut(3, k_3) + \cdots + \mut(r,k_r) \\
	&= \mut(1,j_0) < 0,
	\end{align*}
	contradicting the assumption.
	
	Now we turn to the case that $\sum_{i=1}^r c_i(h) \in \ZZ$ for all $h \in \NN$.
	By the same argument as above, for $1 \leq j \leq \delta_1-1$ there exists an $h \in \NN$ such that 
	$ \sum_{i=1}^r c_i(h) = \mut(1, j) $,	hence $\mut(1,j) \in \ZZ$ for all $j$.
	Further, for each $1 < i \leq r$ and each $1 \leq j \leq \delta_i - 1$, there exists an $h \in \NN$ such that $h \equiv j \mod \delta_i$ and $h \equiv k_\ell \mod \delta_\ell$ 
	for each $1\leq \ell \leq r$, with $\ell \neq i$.
	Thus $\sum_{i=1}^r c_i(h) = \mut(1, j_0) +\mut(i,j)$ for some $j_0$. It follows that $\mut(i,j) \in \ZZ$.
	We conclude that $\tilde{c}_i(h) \in \ZZ$ for all $1\leq i \leq r$ and all $h \in \NN$.
\end{proof}

\begin{lemma}\label{lem:construct}
	Let $c:\NN \rightarrow \QQ$ be a nonnegative periodic function of period $\delta \in \NN$.
	Then for any $\beta \in \NN$ there exists a polynomial $q \in \QQ[t]$ with nonnegative coefficients, such that the coefficient function of the series expansion of 
	\[ \frac{q(t)}{(1-t^\delta)^{\beta}} \]
	is a quasipolynomial of degree $\beta - 1$ whose leading coefficient equals $c$.
\end{lemma}
\begin{proof}
	Write $c = \sum_i c_i e_{\delta,i}$ with $c_i \in \QQ$ nonnegative.
	Recall that the coefficient function of
	\[
	\frac{t^i}{(1-t^\delta)^{\beta}} = \sum_{h \geq 0} \binom{h + \beta - 1}{\beta -1} t^{\delta h+i}
	\] 
	is a quasipolynomial of degree $\beta -1$ with leading coefficient function 
\[
\frac{1}{\delta^{\beta-1}(\beta-1)!} e_{\delta,i}.
\]
	So the polynomial $q(t) \defa \delta^{\beta-1} (\beta-1)! \sum_{i=0}^{\delta-1} c_i t^i$ satisfies the claim.
\end{proof}

\begin{lemma}\label{lem:shift}
	Let $p_1, p_2$ be two quasipolynomials of the same period and the same degree. Assume moreover that the leading coefficient function of $p_1$ is nonnegative and greater than or 
	equal to the leading coefficient function of $p_2$. Then there exists a $k \in \NN$ such that $p_1(h) - p_2(h-k) \geq 0$ for all $h \geq k$.
\end{lemma}
\begin{proof}
	Let $\delta \in \NN$ be the common period of $p_1$ and $p_2$.
	We only consider values of $k$ that are multiples of $\delta$, so we set $k = \kt \delta$.
	Let
	\[ p_1(h) = \sum_{i=0}^\ell a_i(h) h^i \qquad \text{ and } \qquad p_2(h) = \sum_{i=0}^\ell b_i(h) h^i. \]
	Let $\nt \defa h - \kt\delta$. We compute
	\begin{align*}
	p_1(h) - p_2(h-\kt\delta) &= p_1(\nt + \kt\delta) - p_2(\nt) \\
	&= \sum_{i=0}^{\ell} a_i(\nt+\kt\delta) (\nt + \kt\delta)^i - b_i(\nt) \nt^i \\
	&= \sum_{i=0}^{\ell} a_i(\nt) (\nt + \kt\delta)^i - b_i(\nt) \nt^i \\
	&= (a_\ell(\nt) - b_\ell(\nt))\nt^\ell + \sum_{i=0}^{\ell-1} \left( \sum_{j=i}^\ell \binom{j}{i} \kt^{j-i} \delta^{j-i} a_j(\nt) - b_i(\nt) \right) \nt^i.
	\end{align*}
	By assumption we have that $a_\ell(\nt) - b_\ell(\nt) \geq 0$.
	Further, we see that all other coefficient functions of $p_1(\nt + \kt\delta) - p_2(\nt)$ are non-constant polynomials in $\kt$ with leading coefficient $\binom{\ell}{i} \delta^{\ell - i} a_\ell(\nt) > 0$. Therefore all coefficient functions of $p_1(\nt + \kt\delta) - p_1(\nt)$ are nonnegative for $\kt \gg 0$.
	It follows that for a sufficiently large $\kt$, it holds that $p_1(\nt + \kt\delta) - p_2(\nt) \geq 0$ for all $\nt \geq 0$, and
	consequently  $p_1(h) - p_2(h-\kt\delta) \geq 0$ for all $h \geq \kt \delta$.
\end{proof}

Now we are ready to present and prove our main theorem. It shows that a decomposition as in Theorem \ref{thm:jan} is always possible if one allows \emph{rational} coefficients.
\begin{theorem}\label{thm:rat}
	Let $d_1, \dotsc, d_n$ be pairwise coprime or equal positive integer numbers.
	Let $P \in \LS$ be a nonnegative formal Laurent series which is rational with denominator $(1-t^{d_1})\dotsm(1-t^{d_n})$.
	Then it can be written in the form
	\begin{equation*} 
		P(t) = \sum_{I \subseteq \{1, \ldots , n\}} \frac{Q_I(t)}{\prod_{i \in I} \left(1-t^{d_i} \right)} 
	\end{equation*}
	with nonnegative $Q_I \in \QQ[t,t^{-1}]$.
\end{theorem}

Let us introduce some more notation to simplify the presentation of the proof.
Let $\delta_1, \dotsc, \delta_r \in \NN$ denote the different values of the $d_i$, and let $\alpha_i := |\set{j\with d_j = \delta_i}|$ be the multiplicity of $\delta_i$.
Then $P(t)$ is a rational function with denominator $\prod_i (1-t^{\delta_i})^{\alpha_i}$.
From some power of $t$ on, the coefficients of $P$ are given by a quasipolynomial which we denote by $\Q{P}$ (cf.~\cite[Prop.~4.4.1]{stanley}).

\begin{proof}
	We proceed by induction on $\beta \defa \deg \Q{P} + 1$.
	If $\Q{P} = 0$, then $P$ is a polynomial and there is nothing to be proven.
	So from now on we assume that $\Q{P} \neq 0$.
	Using that the $\delta_i$ are pairwise coprime, we compute a partial fraction decomposition of $P$ as follows:
	\begin{align*}
		P(t) &= \frac{p(t)}{(1-t^{\delta_1})^{\alpha_1}\dotsm(1-t^{\delta_r})^{\alpha_r}} 
		= \frac{p(t)}{(1-t)^n \prod_{i=1}^{r}(\sum_{j=0}^{\delta_i-1} t^j)^{\alpha_i}} \\
		&= \frac{p_{0}(t)}{(1-t)^n} + \sum_{i=1}^r \frac{p_{i}(t)}{(\sum_{j=0}^{\delta_i-1} t^j)^{\alpha_i}} \
		= \frac{p_{0}(t)}{(1-t)^n} + \sum_{i=1}^r \frac{p_{i}(t)(1-t)^{\alpha_i}}{(1-t^{\delta_i})^{\alpha_i}};
	\end{align*}
	here, $p, p_0, p_1, \dots, p_r \in \QQ[t,t^{-1}]$.
	Expanding the last expression into a series yields a decomposition
	\begin{equation}\label{eq:dec}
		\Q{P} = q_0 + q_1 + \dotsc + q_r
	\end{equation}
	of $\Q{P}$, where $q_0 \in \QQ[t]$ is a polynomial 
	and $q_i$ is a quasipolynomial of period $\delta_i$ and degree at most $\alpha_i-1$ for $1 \leq i \leq r$.  
	Note that this decomposition is not necessarily unique.
	
	Because $\Q{P}(h)$ is nonnegative for all $h \gg 0$, its leading coefficient $c$ is a nonnegative periodic function.
	There are two cases to distinguish:
	\begin{asparaenum}
		\item
		If $\beta > \max\set{\alpha_i \with 1 \leq i \leq r}$, then $c$ is determined by the first summand in \eqref{eq:dec}.
		In particular, $c$ is a constant function.
		In this case, choose numbers $0 \leq \beta_i \leq \alpha_i$ for $1 \leq i \leq r$ such that $\beta = \beta_1 + \dotsb + \beta_r$.
		Then the coefficient function of the series expansion of $1/\prod_i (1-t^{\delta_i})^{\beta_i}$ is a quasipolynomial of degree $\beta - 1$, and its leading coefficient function is constant.
		Thus there exists a nonnegative $\lambda \in \QQ$ such that $c$ equals the leading coefficient of $\Q{G}$ for
		\[
			G(t) \defa \frac{\lambda}{\prod_{i=1}^r (1-t^{\delta_i})^{\beta_i}}.
		\]
		\item If $\beta \leq \max\set{\alpha_i \with 1 \leq i \leq r}$, then $c$ is a sum of periodic functions of the periods $\delta_i$ for those $i$ where $\beta \leq \alpha_i$.
		By Lemma \ref{lem:periodic}, we can write $c$ as a sum of nonnegative functions $\tilde{c}_1, \dotsc, \tilde{c}_r: \NN \rightarrow \QQ$ of periods $\delta_1, \dotsc, \delta_r$, 
		where $\tilde{c}_i = 0$ if $\beta > \alpha_i$.
		By Lemma \ref{lem:construct}, there are nonnegative polynomials $\tilde{q}_1, \dots, \tilde{q}_r \in \QQ[t]$, such that $c$ is the leading coefficient of $\Q{G}$ for 
		\[
			G(t) \defa \sum_{i=1}^r \frac{\tilde{q}_i(t)}{(1-t^{\delta_i})^{\beta}}.
		\]
	\end{asparaenum}

	In both cases, $\Q{P}$ and $\Q{G}$ satisfy the hypotheses of Lemma \ref{lem:shift}.
	Hence, there exists a $k \in \NN$, such that $ \Q{P}(h) - \Q{G}(h-k) \geq 0$ for all $h \geq k$.
	By enlarging $k$, we may also assume that the coefficient of $t^h$ in $P$ is given by $\Q{P}(h)$ for all $h \geq k$. On the other hand, the coefficients of $G$ are
	given by $\Q{G}(h)$ for all $h \geq 0$.

	Thus, by construction the coefficient of $t^h$ in the series $P' := P - t^k G$ is given by the corresponding coefficient in $P$ for $h < k$ and by $\Q{P}(h) - \Q{G}(h-k)$ for $h \geq k$.
	In particular, $P'$ has nonnegative coefficients.
	But $\deg \Q{P'} < \deg\Q{P}$, so the claim follows by induction.
\end{proof}

\begin{corollary}\label{cor:rat}
	Let $P \in \LS$ be a formal Laurent series satisfying the assumptions of Theorem \ref{thm:rat}. Then there exist a $\lambda \in \NN$ and a finitely generated $R$-module $M$,
	 such that $\lambda P$ is the Hilbert series of $M$.
\end{corollary}
\begin{proof}
	This follows from Theorem \ref{thm:rat} and Theorem \ref{thm:jan}.
\end{proof}

\begin{theorem}\label{thm:2var}
	Assume that in the situation of Theorem \ref{thm:rat} we have $n = 2$.
	Then the numerator polynomials $Q_I$ can be chosen to have nonnegative \emph{integral} coefficients.
	In particular, $P$ can be realized as a Hilbert series of a finitely generated graded $R$-module.
\end{theorem}

\begin{proof} 
As a notation, we write $c_i(P)$ for the $i$-th coefficient of $\Q{P}$.
  If the degrees are equal the problem can be reduced to the standard graded case, so the claim follows from \cite[Thm. 2.1]{U}.
	Therefore we may assume that $d_1 \neq d_2$. Since $n=2$,  $\Q{P}$ has degree at most $1$.
	If $\Q{P}=0$, then $P$ is a polynomial, so nothing is to be proven.
	Next we assume that $\deg \Q{P} = 0$.
	By a partial fraction decomposition of $P$ we see that it can be written in the form
	\begin{equation*}
		P(t) = \frac{p_1(t)}{1-t^{d_1}} + \frac{p_2(t)}{1-t^{d_2}}.
	\end{equation*}
	From this we read off that $c_0(P)$ is the sum of two periodic functions of period $d_1$ resp. $d_2$. By Lemma \ref{lem:periodic}, we can choose these functions to be nonnegative 
	and integer valued. In other words, there exist two polynomials $\tilde{p}_1, \tilde{p}_2 \in \ZZ[t]$ with nonnegative coefficients such that 
	\[ c_0(P) = c_0\left( \frac{\tilde{p}_1(t)}{1-t^{d_1}} + \frac{\tilde{p}_2(t)}{1-t^{d_2}} \right), \]
	so by subtracting a suitable shift of this rational function from $P(t)$ we reduce to the case of a polynomial.
	
	 Finally we consider the case of $\deg \Q{P} = 1$.
	Let us write
	\begin{equation}\label{eq:pq} 
	P(t) = \frac{p(t)}{(1-t^{d_1})(1-t^{d_2})}
	\end{equation}
	with $p(t) \in \QQ[t,t^{-1}]$. First, we show that the coefficients of $p(t)$ are integers.
	For this, let $p(t) = \sum_i a_i t^i$ and write $P(t) = \sum_{j \geq 0} f_j t^j$.
	It follows from \eqref{eq:pq} that 
\[
a_i = f_i - f_{i-d_1} - f_{i-d_2} + f_{i - d_1 d_2} \in \ZZ.
\]
	It is not difficult to see that
	\[ c_1\left(\frac{t^i}{(1-t^{d_1})(1-t^{d_2})}\right) = \frac{1}{d_1 d_2} \]
	for all $i$, and in particular this coefficient function is constant.
	As the coefficients of $p(t)$ are integers, it follows that $c_1(P)$ is an integral multiple of $1/d_1d_2$.
	Hence there exists $\lambda \in \NN$ such that
	\[ P'(t) := P(t) - \frac{\lambda t^k}{(1-t^{d_1})(1-t^{d_2})} \]
	satisfies $\deg \Q{P'} = 0$. Moreover, Lemma \ref{lem:shift} implies that the coefficients of the series expansion of $P'$ are nonnegative for $k \gg 0$. 
	Thus we have reduced the claim to the previous case.
\end{proof}

\section{Counterexamples} \label{section:examples}
The decomposition is not always possible with integral coefficients.
We describe a general construction of counterexamples.
For this we consider pairwise coprime numbers $\delta_1, \dotsc, \delta_r \in \NN$ and exponents $\alpha_1, \dotsc, \alpha_r \in \NN$.
Consider two rational functions $P_1, P_2$ of the form
\[ \frac{1}{\prod_i (1-t^{\delta_i})^{\beta_i}} \]
with $0 \leq \beta_i \leq \alpha_i$.
Assume $P_1$ and $P_2$ have the following properties:
\begin{itemize}
	\item[(i)] $\deg \Q{P_1} = \deg \Q{P_2}$. Let us call this number $d$.
	\item[(ii)] $d+1 > \max \set{\alpha_1,\dotsc,\alpha_r}$. This ensures that the leading coefficients $c_{d}(P_1)$ and $c_{d}(P_2)$ are constant.
	\item[(iii)] $c_d(P_1) > c_d(P_2)$, and the former should not be a multiple of the latter. 
\end{itemize}
Under these assumptions, it is easy to see that there exists a $\lambda \in \NN$, such that
$\tilde{P} := P_1 - \lambda P_2$ is a series, so that $c_d(\tilde{P})$ is smaller than $c_d(P_2)$.
This series may have negative coefficients. But by Lemma \ref{lem:shift} we may instead consider $P := P_1 - \lambda t^k P_2$ for a sufficiently large $k \in \NN$, and this series has nonnegative coefficients.

Now assume additionally that $c_d(P_2)$ is the minimal leading coefficient of all series of the given type and dimension. Then it is immediate that $P$ cannot be written as a nonnegative integral linear combination of such series. We give two explicit examples of this behaviour.

\begin{example} \label{ex3.1}
Consider the rational function
\begin{align*}
P(t) &\defa \frac {1}{(1-t^2)(1-t^5)}-\frac {t^4}{(1-t^3)(1-t^5)} \\
&= \frac{1}{2} \left(1 + t^2 +\frac{t^6}{1-t^2} +\frac{t^2}{1-t^3}+ \frac{1+t^6}{1-t^5} + \frac{t^{12}}{(1-t^3)(1-t^5)} \right).
\end{align*}
One can read off from the first line that the leading coefficient of $\Q{P}$ is $1/10 - 1/15 = 1/30$, and thus smaller than $1/15$.
So by the argument given above, $P(t)$ cannot be written as a nonnegative integral linear combination.
On the other hand, the second line gives a rational decomposition. This shows in particular that the coefficients of the series of $P$ are nonnegative.
\end{example}

\begin{example}
The same phenomenon occurs in the case that there are only two different degrees, say $2$ and $3$, but $\alpha_1, \alpha_2 > 1$. As an explicit example consider the following rational function:
\begin{align*}
P &\defa \frac{1}{(1-t^2)^2 (1-t^3)} - \frac{t^2}{(1-t^2) (1-t^3)^2} \\
&= \frac{1}{2}\left(  \frac{1}{1-t^3}+ \frac{1}{(1-t^2)^2} + \frac{t^3}{(1-t^3)^2} + \frac{t^4}{(1-t^2)(1-t^3)^2} \right).
\end{align*}
\end{example}

\begin{example} \label{ex3.3}
The condition that the degrees $\delta_1, \dotsc, \delta_r$ are pairwise coprime is essential, as the following example shows.
Consider the rational function
\begin{align*}
	P(t) &\defa \frac{1 + t - t^6 - t^{10} - t^{11} - t^{15} + t^{20} + t^{21}}{(1-t^6)(1-t^{10})(1-t^{15})} \\
	&= \frac{1+t+t^7+t^{13}+t^{19}+t^{20}}{1-t^{30}}.
\end{align*}
One can read off from the second line that $P(t)$ cannot be written as a sum with positive coefficients and the required denominator: The coefficient of $t^0$ is $1$, but the terms $t^6, t^{10}$ and $t^{15}$ all have coefficient zero.
\end{example}

\end{document}